\documentclass[12pt]{amsart}%
\usepackage{amsfonts}
\usepackage{amsmath}
\usepackage{amssymb}
\usepackage{graphicx}
\usepackage{xcolor}
\usepackage{hyperref}
\makeatletter
\@addtoreset{equation}{section}
\makeatother

\newtheorem{theorem}{Theorem}[section]

\newtheorem{definition}[theorem]{Definition}

\newtheorem{lemma}[theorem]{Lemma}

\newtheorem{proposition}[theorem]{Proposition}
\newtheorem{remark}[theorem]{Remark}

\begin{document}

\title[Classification for singular Liouville's equation]{Classification of solutions to the singular Liouville's equation associated with the $N$ Finsler Laplacian}

\author{Jianwei Xue and Maochun Zhu}

\address{School of Mathematics and Statistics\\
Nanjing University of Science and Technology\\
Nanjing, 210094, P. R. China\\}

\email{jzxjw88@126.com}

\address{School of Mathematics and Statistics\\
Nanjing University of Science and Technology\\
Nanjing, 210094, P. R. China\\}

\email{zhumaochun2006@126.com}

\thanks{The second author was supported by Natural Science Foundation of China (12571122). {\it Mathematics Subject Classification}. 35J92, 35B53, 35B40, 35B06, {\it Key words and phrases}: Finsler $N$-Laplacian, Classification, Liouville’s equation}

\begin{abstract}
In this paper, we classify a class of singular Liouville's equation associated with the Finsler-$N$-Laplacian for any $\beta\in (0,N)$
\begin{align*}
-\mathrm{div}\left(F^{N-1}(\nabla u)DF(\nabla u)\right)=\hat{F}^{o}(x)^{-\beta}e^u\ \ \text{in } \mathbb{R}^{N}\backslash \{0\},
\end{align*}
under the finite mass condition $\int_{\mathbb{R}^{N}}\hat{F}^{o}(x)^{-\beta}e^u dx<+\infty$. Here $F$ is a convex function, which is positively homogeneous of degree 1, and its polar $F^{o}$ represents a Finsler metric on $\mathbb{R}^{N}$, $\hat{F}^{o}(x)=F^{o}(-x)$. Our result relaxes the mass condition required in the classification result in \cite{LSXZ}.
\end{abstract}

\maketitle

\section{Introduction and main result}
In 1853, Liouville \cite{L} first discovered that solutions to the classical Liouville's equation
\begin{align}\label{e42}
    -\Delta u=e^u\ \ \text{in } \mathbb{R}^{2} 
\end{align}
can be expressed in terms of analytic functions on the complex plane $\mathbb{C}$. It plays an important role in partial differential equation, geometric and complex analysis. In 1991, Chen and Li \cite{CwL} proved, by using the isoperimetric inequality and the method of moving planes, that any solution of \eqref{e42} with the finite mass condition
\begin{align*}
    \int_{\mathbb{R}^{2}}e^{u}dx<+\infty,
\end{align*}
must be radially symmetric and take the explicit form
\begin{align*}
    u(x)=\log \frac{8\lambda^2}{(1+\lambda^2|x-x_0|^2)^2},
\end{align*}
for some $\lambda>0$ and $x_0\in \mathbb{R}^{2}$. We also note that alternative proofs were given by Chou and Wan \cite{CW} using complex analysis and by Chanillo and Kiessling \cite{CK} using a combination of the Pohozaev identity and the isoperimetric inequality. In addition, different techniques were employed to revisit this result in \cite{BLS,HW}.

Subsequently, extensive studies have been devoted to the classification of solutions to the generalized $N$-Liouville’s equation
\begin{align}\label{e43}
\begin{cases}
    -\Delta_N u:=|x|^{-\beta}e^u\ \ \text{in } \mathbb{R}^{N}\\
    \int_{\mathbb{R}^{N}}|x|^{-\beta}e^{u}dx<+\infty
\end{cases}
\end{align}
with integer $N\geq2$, real numbers $\beta<N$ and $\Delta_N u:=\mathrm{div}\left(|\nabla u|^{N-2}\nabla u\right)$.
 
In the case $\beta=0$, Esposito \cite{E} proved that any solution of \eqref{e43} must take an explicit radial form:
 \begin{align*}
    u(x)=\log \frac{(\frac{N}{N-1})^{N-1}(N)^{N}\lambda^N}{(1+\lambda^\frac{N}{N-1}|x-x_0|^\frac{N}{N-1})^N},
\end{align*}
 for some $\lambda>0$ and $x_0\in \mathbb{R}^N$. The proof strategy entails first applying the Kelvin transform to capture the asymptotic behavior at infinity, and then invoking the isoperimetric inequality together with the Pohozaev identity and ODE arguments to achieve the classification. 

When $\beta\neq0$, the situation becomes more delicate  due to the presence of the weighted term $|x|^{-\beta}$. 
 For $N=2$ and $\beta\in (0,2)$, Chen and Li \cite{CwL2} used the method of moving planes to proved that all the solutions of \eqref{e43} are radially symmetric and take the form
\begin{align*}
    u(x)=\log \frac{2(2-\beta)^{2}\lambda^2}{(1+\lambda^2 |x|^{2-\beta})^2},
\end{align*}
for some $\lambda>0$. It should be mentioned that Chanillo and Kiessling \cite{CK2} proved that radial symmetry fails for $\beta= -2k$, $k\in \mathbb{N}\backslash \{0\}$, since non-radial solutions also exist in this regime.  Nevertheless, a complete classification of solutions to \eqref{e43} for arbitrary $\beta<2$ was subsequently obtained by Prajapat and Tarantello \cite{PT}. Their proof crucially relies  on the extension, due to Chou and Wan \cite{CW}, of the Liouville formula \cite{L} to a punctured disc. For $N\geq3$, Esposito \cite{E2} proved that any radial solution of \eqref{e43} takes the form
\begin{align}\label{e52}
    u(x)=\log \frac{(\frac{N}{N-1})^{N-1}(N-\beta)^{N}\lambda^N}{(1+\lambda^\frac{N}{N-1}|x|^\frac{N-\beta}{N-1})^N},
\end{align}
for some $\lambda>0$. Dropping the radial symmetry restriction, Ciraolo, Esposito, and Li \cite{CEL} proposed a $P$-function approach that completely classifies solutions to \eqref{e43} for the full range $\beta\in[0,N)$.  In particular, Ciraolo et al. \cite{CEL} observed that the explicit radial solutions \eqref{e52} are unique for $\beta\in[0,N)$, whereas degeneracy occurs at certain negative values $\beta_k<0$, suggesting non-radial solutions might arise for  $\beta<0$.  Due to a wide range of applications in geometric analysis and partial differential equations (see \cite{Adimurthi04,BM,CJZ-SIAM,Deng14,Dru2,LS,Malchiodi,Ren94}), the classification of Liouville-type equations has been investigated on the half-space, Riemannian manifolds, the Heisenberg groups, etc., see \cite{Berchio23,Catino23,DGHP,Dou22,DHYZ,Lz95,Ma23,Sun25,Z} and the references therein.

In this paper, we are interested in classification of solutions to singular Liouville’s equation associated with the Finsler-$N$-Laplacian. Let $F:\mathbb{R}^{N}\rightarrow
\lbrack 0,+\infty ]$ be a convex function of class $C^2(\mathbb{R}^{N}\backslash \{0\})$ which satisfies the positively homogeneity property:
\begin{align*}
    F(tx)=tF(x),\ \ \forall x\in\mathbb{R}^{N},\ t>0
\end{align*}
and possibly non-symmetric ($F$ is not necessarily centrally symmetric, i.e. $F(x)\neq F(-x)$ in general). Let $F^{o}$ be the support function of $K:=\{x\in \mathbb{R}
^{N}:F(x)\leq 1\}$, which is defined by
\[
F^{o}(x)=\sup_{\xi \neq 0}\frac{\left\langle x,\xi \right\rangle }{
F(\xi )},
\]
and set
\begin{align}\label{e40}
\hat{F}^{o}(x)=F^{o}(-x),\ \ x\in\mathbb{R}^{N}.
\end{align}
The variational structure associated with $F$ gives rise to a family of nonlinear elliptic operators of Finsler type
\begin{align*}
    Q_pu:=\mathrm{div}\left(F^{p-1}(\nabla u)DF(\nabla u)\right),\ \ p>1,
\end{align*}
which we call the Finsler-$p$-Laplacian. These operators naturally appear in the Euler-Lagrange equations of anisotropic energies of the form
\begin{align*}
    \int F^p(\nabla u)dx.
\end{align*}
Anisotropic partial differential equations associated with the Finsler-$p$-Laplacian are of independent interest. Over the past few decades, this research direction has received considerable attention and has been extensively studied, see \cite{BFK,BC,CHN,CFR,DGL,HLYZ,LWY,LSXZ,WX,XG,ZZ} and the references therein.

The analysis of solution structures for such anisotropic problems is substantially more involved than in the Euclidean case, primarily because many classical symmetry properties—including rotation invariance—are no longer available. A fundamental problem in this area is the classification of entire solutions to the Finsler-$p$-Laplacian equation, a topic that has attracted considerable attention in recent literature. For instance, Ciraolo, Figalli and Roncoroni \cite{CFR} provided a complete classification result for the following critical Finsler-$p$-Laplacian equation
\begin{align}\label{e54}
\begin{cases}
    -Q_{p}u=u^{\frac{Np}{N-p}-1}& \text{in } \mathbb{R}^{N}, \\
    u>0& \text{in } \mathbb{R}^{N},\\
    u\in \mathcal{D}^{1,p}(\mathbb{R}^N)
\end{cases}\end{align}
for any $1<p<N$, where the Sobolev space $\mathcal{D}^{1,p}(\mathbb{R}^N)$ is defined as
\begin{align*}
    \mathcal{D}^{1,p}(\mathbb{R}^N) := \left\{ u \in L^{\frac{Np}{N-p}}(\mathbb{R}^N) : \nabla u \in L^p(\mathbb{R}^N) \right\}.
\end{align*}
Based on integral identities and asymptotic estimates, they proved that any solution of equation \eqref{e54} must be of the form
\begin{align*}
u(x)=\left( \frac{\lambda^{\frac{1}{p-1}} \left( N^{\frac{1}{p}} \left( \frac{N-p}{p-1} \right)^{\frac{p-1}{p}} \right)}{\lambda^{\frac{p}{p-1}} + \hat{F}^{o}(x-x_0)^{\frac{p}{p-1}}} \right)^{\frac{N-p}{p}}
\end{align*}
for some $\lambda>0$ and $x_0\in \mathbb{R}^{N}$. When $F$  is the Euclidean norm $| \cdot |$, this classification result recovers the celebrated theorems obtained in \cite{GNN} and \cite{CGS} for $p=2 
$, as well as their extensions to $1<p<N$  developed in \cite{Sc} and \cite{V}. 

Later, Ciraolo and Li \cite{CL} studied the anisotropic $N$-Liouville’s equation, which can be regarded as a counterpart of \eqref{e54} in the case $p=N$. More precisely, they showed that the solution of
\begin{align*}
\begin{cases}
    -Q_{N}u=e^u& \text{in } \mathbb{R}^{N}, \\
    \int_{\mathbb{R}^{N}}e^u dx<+\infty,
\end{cases}
\end{align*}
is completely classified by a family of explicit logarithmic functions:
\begin{align*}
u(x)=\log\frac{N(\frac{N^2}{N-1})^{N-1}\lambda^N}{\left[1+\lambda^{\frac{N}{N-1}}\hat{F}^{o}(x-x_0)^{\frac{N}{N-1}}\right]^N}
\end{align*}
for all $\lambda>0$ and $x_0\in \mathbb{R}^{N}$. Recently, the authors in \cite{LSXZ} studied the following anisotropic singular Liouville's equation
\begin{align}\label{e53}
\begin{cases}
    &-Q_{N}u=\hat{F}^{o}(x)^{-\beta}e^u \text{ in } \mathbb{R}^{N}\backslash \{0\}, \\
   & \int_{\mathbb{R}^{N}}\hat{F}^{o}(x)^{-\beta}e^u dx\leq N(\frac{N(N-\beta)}{N-1})^{N-1}\kappa_N,\\&\sup\limits_{x\in\mathbb{R}^{N}} u(x)< +\infty,
\end{cases}
\end{align}
where $\beta\in (0,N)$ and $\kappa_N$ is the measure of unit Wulff ball (for the precise definition see Section \ref{Anisotropic setting}).  Under this specific mass condition, they obtained the classification results for problem \eqref{e53}  by establishing  an anisotropic weighted isoperimetric inequality to complete the classification of solutions to problem \eqref{e53}. 

In this work, we are devoted to relaxing the mass condition in \cite{LSXZ}, and consider the following anisotropic singular Liouville’s equation with the finite mass condition:
\begin{align}\label{e1}
\begin{cases}
    -Q_{N}u=\hat{F}^{o}(x)^{-\beta}e^u& \text{in } \mathbb{R}^{N}\backslash \{0\}, \\
    \int_{\mathbb{R}^{N}}\hat{F}^{o}(x)^{-\beta}e^u dx<+\infty
\end{cases}
\end{align}
for $\beta\in (0,N)$. Our main result can be stated as following:

\begin{theorem}\label{Th1}
    Suppose $\beta\in (0,N)$. If $u$ is a weak solution
of \eqref{e1}, then it holds
\begin{align*}
u(x)=\log\frac{(\frac{N}{N-1})^{N-1}(N-\beta)^{N}\lambda^N}{\left[1+\lambda^{\frac{N}{N-1}}\hat{F}^{o}(x)^{\frac{N-\beta}{N-1}}\right]^N}
\end{align*}
for some $\lambda>0$.
\end{theorem}

Indeed, if $F(x)=|x|$, then $\hat{F}^{o}(x)=|x|$, and the equation in \eqref{e1} reduces to \eqref{e43}. In the isotropic case, the asymptotic behaviour at infinity for solutions of  Liouville's equation is typically derived through the Kelvin transform (see \cite{E}). However, for general $F$, such a transform is not available as an inversion map. To overcome this difficulty, we will use the strategy developed in \cite{CL}. We first  characterize the logarithmic behavior of a solution $u$ of \eqref{e1} at infinity by scaling argument and test function argument  as follows 
\begin{align}\label{e46}
    u(x)+\gamma\log \hat{F}^{o}(x)=H(x)\in L^{\infty}(\mathbb{R}^{N}\backslash \mathcal{W}_{1}),
\end{align}
where $\gamma$ is a constant given in \eqref{e45}, $\mathcal{W}_{1}$ is unit Wulff ball associated with $\hat{F}^{o}$ (for the precise definition see Section \ref{Anisotropic setting}). After establishing \eqref{e46}, we derive the asymptotic expansion of $\nabla u$  at infinity using a scaling argument motivated by Kichenassamy and V\'{e}ron \cite{KV}. Then, following the approximation argument in \cite{CK,E}, we employ the anisotropic Pohozaev identity, together with the asymptotic behavior of $u$ and $\nabla u$ at infinity, as well as the weighted anisotropic isoperimetric inequality, to show that the level sets of $u$  are Wulff balls centered at the origin. Finally, the classification of solutions to equation \eqref{e1} can be derived from the classification of radial solutions in \cite{E2}.

The organization of this paper is as follows. 

In Section 2, we introduce the required Finsler anisotropic geometry and present some preliminary results, including the logarithmic singularity estimate of solutions at infinity for Finsler-$N$-Laplacian type equations, the weighted anisotropic isoperimetric inequality, the anisotropic Pohozaev identity, among others.

In Section 3, we establish a weighted Brezis-Merle type inequality concerning the difference of two functions, and investigate the asymptotic behavior at infinity of both the solution $u$  and its gradient $\nabla u$.

 In Section 4, we first employ the anisotropic Pohozaev identity to establish the quantization result for the mass $\int_{\mathbb{R}^{N}}\hat{F}^{o}(x)^{-\beta}e^u dx$, thereby completing the proof of Theorem \ref{Th1}

Throughout this paper, the letters $C$ always denotes some positive constants
which may vary from line to line.

\section{Finsler anisotropic geometry and some preparatory results}

\subsection{Finsler anisotropic geometry}\label{Anisotropic setting}

We shall assume that $Hess(F^2)$ is positive definite in $\mathbb{R}^{N}\backslash\{0\}$. Then by
Xie and Gong \cite{XG}, $Hess(F^N)$ is also positive definite in $\mathbb{R}^{N}\backslash\{0\}$. Based on the homogeneity of $F$, we have
\begin{align}\label{e39}
\alpha |\xi |\leq F(\xi )\leq \eta |\xi |,\ \ \forall \ \xi \in \mathbb{R}^{N},
\end{align}
for some positive constants $\alpha \leq \eta $. According to the definition of $F^{o}$, it is not difficult to check that $F^{o}:\mathbb{R}^{N}\rightarrow
\lbrack 0,+\infty ]$ is a convex, positively homogeneous function, and satisfies
\begin{align*}
\frac{1}{\alpha} |x |\geq F^{o}(x)\geq \frac{1}{\eta} |x |,\ \ \forall \ x \in \mathbb{R}^{N},
\end{align*}
and $F(x)=\sup\limits_{\xi \neq 0}\frac{\left\langle x,\xi \right\rangle }{F^{o }(\xi )}$. Similar to \eqref{e40}, we set
\begin{align*}
\hat{F}(\xi)=F(-\xi),\ \ \xi\in\mathbb{R}^{N}.
\end{align*}
Define the unit Wulff ball centered at the origin
\begin{align*}
\mathcal{W}_{1}(0):=\{x\in \mathbb{R}^{N}:\hat{F}^{o}(x)\leq 1\}
\end{align*}
and denote by $\kappa_N$ the measure of $\mathcal{W}_{1}(0)$. The Wulff ball in the metric $\hat{F}^{o}$ centered at $x_0$ with the radius $r$ is given by
\begin{align*}
\mathcal{W}_{r}(x_0):=\{x\in \mathbb{R}^{N}:\hat{F}^{o}(x-x_0)\leq r\}.
\end{align*}
Obviously, $|\mathcal{W}_{r}(x_0)|=\kappa_N r^N$. For simplicity, throughout the paper we use the notation $\mathcal{W}_{r}:=\mathcal{W}_{r}(0)$ for any $r>0$. The study of the Wulff ball $\mathcal{W}_{r}(x_0)$ $(x_0\in\mathbb{R}^{N},r>0)$ was initiated in Wulff’s work \cite{W} on crystal shapes and minimization of a surface functional $\int_{\partial\Omega}\hat{F}(\nu)d\sigma(x)$ among regular domains $\Omega$, here $\sigma(x)$ denotes the $n-1$ dimension Hausdorff measure and $\nu$ is the outer unit normal on $\partial\Omega$.

Next, we collect some important properties of the function $F(x)$ and its polar $%
F^{o}(x)$ which will be used in the following:

\begin{lemma}
\label{L2.1}The function $F(x)$ and its polar $F^{o}(x)$ satisfy:

$(i)$ $|F(x)-F(y)|\leq F(x+y)\leq F(x)+F(y);$

$(ii)$ $\frac{1}{C}\leq|\nabla F(x)|\leq C,$ and $\frac{1}{C}\leq|\nabla
F^{o}(x)|\leq C$ for some $C>0$ and any $x\neq0;$

$(iii)$ $\left\langle x,\nabla F(x)\right\rangle =F(x),$ $\left\langle
x,\nabla F^{o}(x)\right\rangle =F^{o}(x)$ for any $x\neq0;$

$(iv)$ $F(\nabla F^{o}(x))=1,$ $F^{o}(\nabla F(x))=1$ for any $%
x\neq0;$

$(v)$ $F^{o}(x)DF(\nabla F^{o}(x))=x$ for any $x\neq 0;$

$(vi)$ $DF(t\xi )=DF(\xi )$ for any $\xi \neq 0$ and $t> 0.$
\end{lemma}

\begin{remark}
    The function $\hat{F}(x)$ and its polar $\hat{F}^{o}(x)$ also satisfy Lemma \ref{L2.1}.
\end{remark}

\subsection{Anisotropic auxiliary results}

In this subsection, we will present some results concerning anisotropic equations, weighted anisotropic isoperimetric inequalities and anisotropic Pohozaev identity, which will be applied to investigating the asymptotic behavior of solutions at infinity and to estimating for mass $\int_{\mathbb{R}^{N}}\hat{F}^{o}(x)^{-\beta}e^u dx$.

First, according to Theorem 2 in \cite{S1} (see also \cite[Lemma 2.4]{CL}), we can obtain the following regularity result.

\begin{lemma}\label{L6}
    Let $u$ be a weak solution of \eqref{f} defined in some open wulff ball $\mathcal{W}_{2R}\subset \Omega$, $f\in L^{\frac{N}{N-\epsilon}}(\Omega)$ for some $0<\epsilon\leq 1$. Then
\begin{align*}
    \vert\vert u^+\vert\vert_{L^{\infty}(\mathcal{W}_{R})}\leq C(R)\vert\vert u^+\vert\vert_{L^{N}(\mathcal{W}_{2R})}+C(R)\vert\vert f\vert\vert^{\frac{1}{N-1}}_{L^{\frac{N}{N-\epsilon}}(\mathcal{W}_{2R})}.
\end{align*}
\end{lemma}

In the following, we introduce a result concerning logarithmic singularity estimates of nonnegative solutions at infinity for a class of Finsler-$N$-Laplacian type equations.

\begin{theorem}\label{Th2}
    Let $u$ be a non-negative solution of
\begin{align*}
    \mathrm{div}\left(\hat{F}^{N-1}(\nabla u)D\hat{F}(\nabla u)\right)=f(x)\ \ \text{in\ }\mathbb{R}^{N}\backslash \mathcal{W}_{1},
\end{align*}
and assume that 
\begin{align*}
0\leq f(x)\leq\frac{C}{\hat{F}^{o}(x)^{N}},\ \ x\in \mathbb{R}^{N}\backslash \mathcal{W}_{1}
\end{align*}
holds. If $u(x)\rightarrow +\infty$ as $\hat{F}^{o}(x)\rightarrow +\infty$, then there exists a constant $d>1$ such that
\begin{align*}
\frac{1}{d}\log \hat{F}^{o}(x)\leq u(x)\leq d\log \hat{F}^{o}(x)
\end{align*}
holds in the neighborhood of infinity.
\end{theorem}
\begin{proof}
    According to Lemma \ref{L2.1}-(ii), (iv) and the homogeneity of $\hat{F}^{o}(x)$, this result immediately follows by Theorem 6.1 in \cite{CL}.
\end{proof}

Next, we give the following rigidity result which showed in Lemma 3.2 in \cite{CL2022} (see also Lemma 6.6 in \cite{CL}).

\begin{lemma}
\label{L4}Assume that $H(x)\in C^{1}(\mathbb{R}^{N}\backslash \{0\})\cap L^{\infty}(\mathbb{R}^{N})$ satisfies
\begin{align*}
Q_{N}(\gamma \log F^{o}(x)+H(x))=0\ \ \text{or}\ \ Q_{N}(-\gamma \log \hat{F}^{o}(x)+H(x))=0
\end{align*}
in $\mathbb{R}^{N}\backslash \{0\}$, for some constant $\gamma\geq 0$. Then $H(x)$ is a constant function.
\end{lemma}

\begin{lemma}(\cite[Corollary 2.5]{LSXZ}) \label{L3}Suppose $\beta\in(0,N)$. Then for any measurable set $\Omega$ in $\mathbb{R}^{n}$,
\begin{equation}
\frac{\int_{\partial \Omega}\hat{F}^{o}(x)^{-\frac{N-1}{N}\beta}\hat{F}(\nu)d\sigma(x)}{\left(  \int_{\Omega}\hat{F}^{o}(x)^{-\beta}dx\right)
^{(N-1)/N}}\geq(N\kappa_{N})^{1/N}\times(N-\beta)^{(N-1)/N},
\label{eq1.3}
\end{equation}
where $\nu$ is the outer unit normal on $\partial \Omega$. Equality in (\ref{eq1.3})
holds if and only if $\Omega$ is a Wulff ball centered at the origin.
\end{lemma}

In the case that $\hat{F}^{o}(x)$ is even, i.e. $\hat{F}^{o}(x)=F^{o}(x)$, the above weighted anisotropic isoperimetric inequalities was proved in \cite{LSXZ}. Indeed, the proof of Theorem 1.1 in \cite{LSXZ} can be adapted to establish that \eqref{eq1.3} holds for any $\hat{F}^{o}(x)$, possibly non-symmetric. 

Finally, we present the following anisotropic Pohozaev identity, which can be regarded as a direct corollary of the Theorem 1.2 in \cite{Montoro 2023}.  For this, we recall the definition of weak solutions.

\begin{definition}
    We say that $u$ is a weak solution of equation 
     $$ -Q_{N}u=f(x)\ \ \text{in }\Omega\subset\mathbb{R}^{N},$$
 if $u\in W_{loc}^{1,N}(\Omega)$ and
\begin{align*}
    \int_{\mathbb{R}^{N}}F^{N-1}(\nabla u)\left \langle DF(\nabla u),\nabla\varphi\right\rangle dx=\int_{\mathbb{R}^{N}}f(x) \varphi dx
\end{align*}
for all $\varphi\in W_{0}^{1,N}(\Omega)$ with $\Omega\subset\mathbb{R}^{N}$ bounded.
\end{definition}

\begin{lemma}\label{L5}Let $\Omega$ a bounded smooth domain and let $u\in C^1(\overline{\Omega})$ be a weak solution
to
\begin{align*}
-Q_{N}u=\hat{F}^{o}(x)^{-\beta}e^u\ \ \text{in }\Omega.
\end{align*}
Then we have
\begin{align*}
&(N-\beta)\int_{\Omega}\hat{F}^{o}(x)^{-\beta}e^{u}dx-(N-\beta)\int_{\Omega}\hat{F}^{o}(x)^{-\beta}dx\\
=&\int_{\partial \Omega}\hat{F}^{o}(x)^{-\beta}e^{u}\langle x,\nu \rangle d\sigma(x)-\int_{\partial \Omega}\hat{F}^{o}(x)^{-\beta}\langle x,\nu \rangle d\sigma(x)-\frac{1}{N} \int_{\partial \Omega}F^{N}(\nabla u)\langle x,\nu \rangle d\sigma(x)\\
&+\int_{\partial \Omega}F^{N-1}(\nabla u)\langle DF(\nabla u),\nu \rangle\langle x,\nabla u \rangle d\sigma(x),
\end{align*}
where $\nu$ is the outer unit normal on $\partial \Omega$.
\end{lemma}

\section{Asymptotic Behavior of Solutions and Their Gradients at Infinity}
In this section, we shall determine precisely the asymptotic behavior at infinity for the solution $u$ and $\nabla u$.

First, we need the following weighted Brezis-Merle type inequality which associates the difference of two functions. This type of exponential inequality was first obtained by Brezis and Merle in \cite{BM}. When $\beta=0$, corresponding anisotropic version results in dimension two was considered in \cite[Theorem 5.2]{WX}, and the higher dimensional case was studied in \cite[Theorem 1.2]{XG} (see also \cite{AA,RW} for isotropic case).
\begin{theorem}\label{L1}
    Suppose $\beta\in(0,N)$. Let $\Omega$ be a bounded domain in $\mathbb{R}^{N}$ and $f\in L^1(\Omega)$. Suppose that $u$ and $v$ are the weak solutions of
\begin{align}\label{f}
-Q_{N}u=f(x)>0\text{\ \ in\ }\Omega
\end{align}
and
\begin{align}\label{0}
\begin{cases}
    -Q_{N}v=0& \text{in } \Omega, \\
    v=u& \text{on } \partial\Omega,
\end{cases}
\end{align}
respectively. Then for every $0<\lambda<\delta\vert\vert f\vert\vert_{L_1(\Omega)}^{-\frac{1}{N-1}}$, it holds
\begin{align*}
\int_{\Omega}\hat{F}^{o}(x)^{-\beta}e^{\lambda|u-v|}dx\leq\frac{\int_{\Omega}\hat{F}^{o}(x)^{-\beta}dx}{1-\lambda\delta^{-1}\vert\vert f\vert\vert_{L_1(\Omega)}^{\frac{1}{N-1}}},
\end{align*}
where
\begin{align*}
\delta:=(N-\beta)(N\kappa_N)^{\frac{1}{N-1}} d_0^{\frac{1}{N-1}}
\end{align*}
and
\begin{align*}
d_0=\inf\{d_{X,Y}:X,Y\in \mathbb{R}^{N},X\neq0,Y\neq0,X\neq Y\}
\end{align*}
with
\begin{align*}
d_{X,Y}=\frac{\left \langle F^{N-1}(X)DF(X)-F^{N-1}(Y)DF(Y),X-Y \right \rangle}{F^{N}(X-Y)}.
\end{align*}
\end{theorem}

\begin{proof}
    Let $\Omega_t=\{x\in \Omega:|u-v|>t\}$, $\mu_\beta(t)=\int_{\Omega_t}\hat{F}^{o}(x)^{-\beta}dx$ and $\lambda_\beta:=(N\kappa_{N})^{1/N}\times(N-\beta)^{(N-1)/N}$. Taking the difference between \eqref{f} and \eqref{0}, by the Fundamental Theorem of Calculus, we have
\begin{align}\label{e2}
0<f=-(Q_{N}u-Q_{N}v)=-\tilde{Q}_{N}(u-v),
\end{align}
where
\begin{align*}
\tilde{Q}_{N}(u-v)=\sum_{i,j}\partial_{x_i}\left[\int_0^1\left(\frac{1}{N}F^N_{\xi_i \xi_j}(t\nabla u+(1-t)\nabla v)\right)dt\partial_{x_j}(u-v)\right].
\end{align*}
Since $Hess(F^N)$ is positive definite, $\tilde{Q}$ is an elliptic operator. Hence, we can apply Hopf’s boundary lemma to the uniformly elliptic equation \eqref{e2} to get
\begin{align*}
\frac{\partial (u-v)}{\partial \nu}<0\ \ \text{on}\ \partial\Omega_t.
\end{align*}
It follows from \eqref{e2}, divergence theorem and the definition of $d_0$ that
\begin{align}
\int_{\Omega_t}f(x)dx&=\int_{\Omega_t}-(Q_{N}u-Q_{N}v)dx\nonumber\\
&=\int_{\partial\Omega_t}\left \langle F^{N-1}(\nabla u)DF(\nabla u)-F^{N-1}(\nabla v)DF(\nabla u),\frac{\nabla u-\nabla v}{|\nabla u-\nabla v|} \right \rangle d\sigma(x)\nonumber\\
&\geq d_0\int_{\partial\Omega_t}\frac{F^{N}(\nabla u-\nabla v)}{|\nabla u-\nabla v|}d\sigma(x)\label{e47}.
\end{align}
By the weighted isoperimetric inequality \eqref{eq1.3}, the co-area formula and H\"{o}lder's inequality, it follows that
\begin{align*}
\lambda_\beta\mu_\beta(t)^{\frac{N-1}{N}}&\leq\int_{\partial \Omega_t}\hat{F}^{o}(x)^{-\frac{N-1}{N}\beta}\hat{F}(\nu)d\sigma(x)\\
&=\int_{\partial\Omega_t}\hat{F}^{o}(x)^{-\frac{N-1}{N}\beta}\frac{F(\nabla u-\nabla v)}{|\nabla u-\nabla v|} d\sigma(x)\\
&\leq \left(\int_{\partial\Omega_t}\frac{F^{N}(\nabla u-\nabla v)}{|\nabla u-\nabla v|}d\sigma(x)\right)^{\frac{1}{N}}\left(\int_{\partial\Omega_t}\frac{\hat{F}^{o}(x)^{-\beta}}{|\nabla u-\nabla v|}d\sigma(x)\right)^{\frac{N-1}{N}}\\
&=\left(\int_{\partial\Omega_t}\frac{F^{N}(\nabla u-\nabla v)}{|\nabla u-\nabla v|}d\sigma(x)\right)^{\frac{1}{N}}\left(-\mu_\beta'(t)\right)^{\frac{N-1}{N}},
\end{align*}
which combined with \eqref{e47} yields
\begin{align*}
-\mu_\beta'(t)\geq\frac{\lambda_\beta^{\frac{N}{N-1}}d_0^{\frac{1}{N-1}}\mu_\beta(t)}{\left(\int_{\Omega_t}f(x)dx\right)^{\frac{1}{N-1}}}.
\end{align*}
Hence,
\begin{align*}
-\frac{dt}{d\mu_\beta}\leq \frac{\vert\vert f\vert\vert_{L_1(\Omega)}^{\frac{1}{N-1}}}{\lambda_\beta^{\frac{N}{N-1}}d_0^{\frac{1}{N-1}}\mu_\beta}.
\end{align*}
Integrating the last inequality over $(\mu_\beta,\mu_\beta(0))$, we deduce
\begin{align*}
t(\mu_\beta)\leq \frac{\vert\vert f\vert\vert_{L_1(\Omega)}^{\frac{1}{N-1}}}{\lambda_\beta^{\frac{N}{N-1}}d_0^{\frac{1}{N-1}}}\log(\frac{\mu_\beta(0)}{\mu_\beta}),
\end{align*}
therefore for $0<\epsilon<1$,
\begin{align*}
\exp{\left(\frac{(1-\epsilon)\lambda_\beta^{\frac{N}{N-1}}d_0^{\frac{1}{N-1}}}{\vert\vert f\vert\vert_{L_1(\Omega)}^{\frac{1}{N-1}}}t(\mu_\beta)\right)}\leq (\frac{\mu_\beta(0)}{\mu_\beta})^{1-\epsilon}.
\end{align*}
Using the co-area formula again, we have by integrating above inequality that
\begin{align*}
&\int_\Omega\hat{F}^{o}(x)^{-\beta}\exp{\left[\frac{(1-\epsilon)\lambda_\beta^{\frac{N}{N-1}}d_0^{\frac{1}{N-1}}}{\vert\vert f\vert\vert_{L_1(\Omega)}^{\frac{1}{N-1}}}|u-v|\right]}dx\\
=&\int_0^{\infty} \exp{\left[\frac{(1-\epsilon)\lambda_\beta^{\frac{N}{N-1}}d_0^{\frac{1}{N-1}}t}{\vert\vert f\vert\vert_{L_1(\Omega)}^{\frac{1}{N-1}}}\right]}\int_{\partial\Omega_t}\frac{\hat{F}^{o}(x)^{-\beta}}{|\nabla u-\nabla v|}d\sigma(x)dt\\
=&\int_0^{\infty} \exp{\left[\frac{(1-\epsilon)\lambda_\beta^{\frac{N}{N-1}}d_0^{\frac{1}{N-1}}t}{\vert\vert f\vert\vert_{L_1(\Omega)}^{\frac{1}{N-1}}}\right]}\left(-\mu_\beta'(t)\right)dt\\
=&\int_0^{\mu_\beta(0)} \exp{\left[\frac{(1-\epsilon)\lambda_\beta^{\frac{N}{N-1}}d_0^{\frac{1}{N-1}}t(\mu_\beta)}{\vert\vert f\vert\vert_{L_1(\Omega)}^{\frac{1}{N-1}}}\right]}d\mu_\beta\leq\frac{\mu_\beta(0)}{\epsilon}.
\end{align*}
Letting 
\[
    \lambda=\frac{(1-\epsilon)\lambda_\beta^{\frac{N}{N-1}}d_0^{\frac{1}{N-1}}}{\vert\vert f\vert\vert_{L_1(\Omega)}^{\frac{1}{N-1}}},
\]
we get
\begin{align*}
\int_{\Omega}\hat{F}^{o}(x)^{-\beta}e^{\lambda|u-v|}dx\leq\frac{\mu_\beta(0)}{1-\lambda\delta^{-1}\vert\vert f\vert\vert_{L_1(\Omega)}^{\frac{1}{N-1}}}.
\end{align*}
The proof is finished.
\end{proof}

Using this weighted Brezis-Merle type inequality, one can get the regularity result of solution $u$ to \eqref{e1} and the globally logarithmic estimate.
\begin{proposition}\label{P3.1}
    Let $u$ be a solution of \eqref{e1}. Then $u^+\in L^{\infty}(\mathbb{R}^{N})$ and $u\in C_{loc}^{1,\alpha}(\mathbb{R}^{N}\backslash \{0\})$ for some $\alpha\in (0,1)$.
\end{proposition}

\begin{proof}
    Let $\bar{x}\in \mathbb{R}^{N}$ and $r>0$, such that $\mathcal{W}_{r}(\bar{x})\subset\mathbb{R}^{N}$. Suppose that $v$ is the weak solution of
\begin{align*}
\begin{cases}
    -Q_{N}v=0& \text{in } \mathcal{W}_{r}(\bar{x}), \\
    v=u& \text{on } \partial\mathcal{W}_{r}(\bar{x}).
\end{cases}
\end{align*}
In view of the comparison principle, we obtain $v\leq u$ in $\mathcal{W}_{r}(\bar{x})$, i.e. $v^+\leq u^+$ in $\mathcal{W}_{r}(\bar{x})$. Then we have
\begin{align}
    \int_{\mathcal{W}_{r}(\bar{x})}(v^+)^{N}dx&\leq \int_{\mathcal{W}_{r}(\bar{x})}(u^+)^{N}dx\leq
    N!\int_{\mathcal{W}_{r}(\bar{x})}e^{u}dx\leq
    N!C(r)\int_{\mathcal{W}_{r}(\bar{x})}\hat{F}^{o}(x)^{-\beta}e^{u}dx\nonumber\\
    &\leq N!C(r)\int_{\mathbb{R}^{N}}\hat{F}^{o}(x)^{-\beta}e^{u}dx<+\infty\label{e3}.
\end{align}
Applying Serrin’s local $L^{\infty}$ estimate to $v^+$ (Lemma \ref{L6}), we get
\begin{align}\label{e4}
    \vert\vert v^+\vert\vert_{L^{\infty}(\mathcal{W}_{\frac{r}{2}}(\bar{x}))}\leq C(r)\vert\vert v^+\vert\vert_{L^{N}(\mathcal{W}_{r}(\bar{x}))}\leq C(r).
\end{align}

Fix $\epsilon>0$ satisfies $\epsilon<N-\beta$. Assume that $r$ small enough, such that
\begin{align*}
    \vert\vert \hat{F}^{o}(x)^{-\beta}e^u\vert\vert^{\frac{1}{N-1}}_{L^{1}(\mathcal{W}_{r}(\bar{x}))}\leq \frac{(N-\beta-\epsilon)\delta}{3(N-\beta)},
\end{align*}
where $\delta$ is defined in Theorem \ref{L1}. Taking
\[
    q=\frac{2(N-\beta)(N-\epsilon)}{N(N-\beta-\epsilon)}>1.
\]
Applying H\"{o}lder's inequality yields
\begin{align}\label{e6}
&\int_{\mathcal{W}_{r}(\bar{x})}\hat{F}^{o}(x)^{-\frac{N}{N-\epsilon}\beta}e^{\frac{N}{N-\epsilon} (u-v)}dx\nonumber\\
=&\int_{\mathcal{W}_{r}(\bar{x})}\hat{F}^{o}(x)^{-\frac{N}{N-\epsilon}\beta+\frac{\beta}{q}}\hat{F}^{o}(x)^{-\frac{\beta}{q}}e^{\frac{N}{N-\epsilon} (u-v)}dx\nonumber\\
\leq&\left(\int_{\mathcal{W}_{r}(\bar{x})}\hat{F}^{o}(x)^{-\frac{N}{N-\epsilon}\frac{q}{q-1}\beta+\frac{\beta}{q-1}}dx\right)^{\frac{q-1}{q}}\left(\int_{\mathcal{W}_{r}(\bar{x})}\hat{F}^{o}(x)^{-\beta}e^{\frac{qN}{N-\epsilon} (u-v)}dx\right)^{\frac{1}{q}}.
\end{align}
It is easy to see that
\[
    -\frac{N}{N-\epsilon}\frac{q}{q-1}\beta+\frac{\beta}{q-1}>-N,
\]
which implies that the integral
\begin{align}\label{e49}
\int_{\mathcal{W}_{r}(\bar{x})}\hat{F}^{o}(x)^{-\frac{N}{N-\epsilon}\frac{q}{q-1}\beta+\frac{\beta}{q-1}}dx\leq C(r).
\end{align}
Notice that
\[
    \frac{qN}{N-\epsilon}=\frac{2(N-\beta)}{N-\beta-\epsilon}<\frac{3(N-\beta)}{N-\beta-\epsilon}\leq \delta\vert\vert \hat{F}^{o}(x)^{-\beta}e^u\vert\vert^{-\frac{1}{N-1}}_{L^{1}(\mathcal{W}_{r}(\bar{x}))}.
\]
According to Theorem \ref{L1}, we obtain
\begin{align}\label{e48}
\int_{\mathcal{W}_{r}(\bar{x})}\hat{F}^{o}(x)^{-\beta}e^{\frac{qN}{N-\epsilon} (u-v)}dx\leq C(r).
\end{align}
Therefore, from \eqref{e49} and \eqref{e48}, we know that
\begin{align*}
\int_{\mathcal{W}_{r}(\bar{x})}\hat{F}^{o}(x)^{-\frac{N}{N-\epsilon}\beta}e^{\frac{N}{N-\epsilon} (u-v)}dx\leq C(r).
\end{align*}
This together with \eqref{e4} implies
\begin{align}\label{e5}
\int_{\mathcal{W}_{\frac{r}{2}}(\bar{x})}\hat{F}^{o}(x)^{-\frac{N}{N-\epsilon}\beta}e^{\frac{N}{N-\epsilon} u}dx=\int_{\mathcal{W}_{\frac{r}{2}}(\bar{x})}\hat{F}^{o}(x)^{-\frac{N}{N-\epsilon}\beta}e^{\frac{N}{N-\epsilon}(u-v)}e^{\frac{N}{N-\epsilon} v}dx\leq C(r).
\end{align}
Thanks to \eqref{e3} and \eqref{e5}, applying Serrin’s local $L^{\infty}$ estimate (Lemma \ref{L6}) to $u^+$, we have
\begin{align}\label{e7}
    \vert\vert u^+\vert\vert_{L^{\infty}(\mathcal{W}_{\frac{r}{4}}(\bar{x}))}\leq C(r)\vert\vert u^+\vert\vert_{L^{N}(\mathcal{W}_{\frac{r}{2}}(\bar{x}))}+C(r)\vert\vert \hat{F}^{o}(x)^{-\beta}e^{u}\vert\vert_{L^{\frac{N}{N-\epsilon}}(\mathcal{W}_{\frac{r}{2}}(\bar{x}))}\leq C(r).
\end{align}

Let $R_0>0$. For any $\bar{x}\in \mathbb{R}^{N}\backslash\overline{\mathcal{W}_{R_0+1}}$ and assume that $r=1$, as long as $R_0$ is sufficiently large, it holds
\begin{align*}
    \vert\vert \hat{F}^{o}(x)^{-\beta}e^u\vert\vert^{\frac{1}{N-1}}_{L^{1}(\mathcal{W}_{1}(\bar{x}))}\leq \vert\vert \hat{F}^{o}(x)^{-\beta}e^u\vert\vert^{\frac{1}{N-1}}_{L^{1}(\mathbb{R}^{N}\backslash{\mathcal{W}_{R_0}})}\leq\frac{(N-\beta-\epsilon)\delta}{3(N-\beta)}.
\end{align*}
Hence, it follows from \eqref{e7} that $\vert\vert u^+\vert\vert_{L^{\infty}(\mathbb{R}^{N}\backslash\overline{\mathcal{W}_{R_0+1}})}\leq C(1)$.
Due to $\overline{\mathcal{W}_{R_0+1}}$ is compact, one can find a finite number of points $x_i$ and numbers $r_i$ $(i=1,...,m)$ satisfying 
\begin{align*}
    \vert\vert \hat{F}^{o}(x)^{-\beta}e^u\vert\vert^{\frac{1}{N-1}}_{L^{1}(\mathcal{W}_{r_i}(x_i))}\leq\frac{(N-\beta-\epsilon)\delta}{3(N-\beta)}
\end{align*}
such that $\overline{\mathcal{W}_{R_0+1}}\subset\bigcup_{i=1}^{m}\mathcal{W}_{r_i/4}(x_i)$. By \eqref{e7} again, we have $\vert\vert u^+\vert\vert_{L^{\infty}(\overline{\mathcal{W}_{R_0+1}})}\leq +\infty$. Therefore, we obtain that $u^+\in L^{\infty}(\mathbb{R}^{N})$ and $\hat{F}^{o}(x)^{-\beta}e^u\in L^{\infty}(\mathbb{R}^{N}\backslash \{0\})$. We can apply the classical regularity results in \cite{D,S1,T} to deduce that $u\in C^{1,\alpha}_{loc}(\mathbb{R}^{N}\backslash \{0\})$ for some $\alpha\in (0,1)$.
\end{proof}

\begin{proposition}\label{P2}
    Let $u$ be a solution of \eqref{e1}. Then there exists a constant $C>0$ such that
\begin{align}\label{e9}
    u(x)\leq C-(N-\beta)\log \hat{F}^{o}(x)\ \ \text{for\ any\ }x\in \mathbb{R}^{N}\backslash\{0\}.
\end{align}
\end{proposition}

\begin{proof}
For $R>0$, denote by
\begin{align*}
    u_R(x)=u(Rx)+(N-\beta)\log R,\ \ \forall x\in \mathbb{R}^{N}.
\end{align*}
Clearly, $u_R$ is also a solution of \eqref{e1}. Given $x_0\in\partial\mathcal{W}_{1}$. Let $v_R$ is the weak solution of
\begin{align*}
\begin{cases}
    -Q_{N}v_R=0& \text{in } \mathcal{W}_{\frac{1}{2}}(x_0), \\
    v_R=u_R& \text{on } \partial\mathcal{W}_{\frac{1}{2}}(x_0).
\end{cases}
\end{align*}
Using the similar argument as \eqref{e3} and \eqref{e4}, one has
\begin{align*}
    \int_{\mathcal{W}_{\frac{1}{2}}(x_0)}(v_R^+)^{N}dx&\leq \int_{\mathcal{W}_{\frac{1}{2}}(x_0)}(u_R^+)^{N}dx\leq N!C\int_{\mathbb{R}^{N}}\hat{F}^{o}(x)^{-\beta}e^{u_R}dx\\
    &= N!C\int_{\mathbb{R}^{N}}\hat{F}^{o}(x)^{-\beta}e^{u}dx<+\infty,
\end{align*}
and
\begin{align*}
    \vert\vert v_R^+\vert\vert_{L^{\infty}(\mathcal{W}_{\frac{1}{4}}(x_0))}\leq C.
\end{align*}

Since
\begin{align*}
    \int_{\mathcal{W}_{\frac{1}{2}}(x_0)}\hat{F}^{o}(x)^{-\beta}e^{u_R}dx\leq\int_{\mathbb{R}^{N}\backslash\mathcal{W}_{\frac{1}{2}}}\hat{F}^{o}(x)^{-\beta}e^{u_R}dx=\int_{\mathbb{R}^{N}\backslash\mathcal{W}_{\frac{R}{2}}}\hat{F}^{o}(x)^{-\beta}e^{u}dx,
\end{align*}
there exists $\bar{R}$ big enough such that for any $R>\bar{R}$, we have
\begin{align*}
    \vert\vert \hat{F}^{o}(x)^{-\beta}e^{u_R}\vert\vert^{\frac{1}{N-1}}_{L^{\infty}(\mathcal{W}_{\frac{1}{2}}(x_0))}\leq \frac{(N-\beta-\epsilon)\delta}{3(N-\beta)}.
\end{align*}
Using the similar arguments as in \eqref{e6}-\eqref{e7}, we obtain
\begin{align*}
    \vert\vert u_R^+\vert\vert_{L^{\infty}(\mathcal{W}_{\frac{1}{8}}(x_0))}\leq C,
\end{align*}
for any $R>\bar{R}$. This implies
\begin{align*}
    u_R^+(x)<C\ \ \text{on\ }\partial\mathcal{W}_{1},
\end{align*}
in view of the definitions of $u_R$, we see that
\begin{align*}
    u(Rx)+(N-\beta)\log R<C\ \ \text{on\ }\partial\mathcal{W}_{1}.
\end{align*}
Hence, 
\begin{align}\label{e8}
    u(x)+(N-\beta)\log \hat{F}^{o}(x)<C,
\end{align}
for all $x\in \mathbb{R}^{N}\backslash\mathcal{W}_{\bar{R}}$. Since $u^+\in L^{\infty}(\mathbb{R}^{N})$ from Proposition \ref{P3.1}, it is easy to see that \eqref{e8} holds for $x\in \overline{\mathcal{W}_{\bar{R}}}\backslash\{0\}$. Then we accomplish the proof.
\end{proof}

In the following, we will study the limit behavior for solution $u$ of \eqref{e1} and $\nabla u$ at infinity. For the convenience, we set
\begin{align}\label{e50}
    \gamma_0:=\left[\frac{\int_{\mathbb{R}^{N}}\hat{F}^{o}(x)^{-\beta}e^u dx}{N\kappa_N}\right]^{\frac{1}{N-1}}.
\end{align}
\begin{proposition}\label{P4}
    Let $u$ be a solution of \eqref{e1}. Then
\begin{align}\label{e33}
    u(x)+\gamma_0\log \hat{F}^{o}(x)\in L^{\infty}(\mathbb{R}^{N}\backslash \mathcal{W}_{1}),
\end{align}
and
\begin{align}\label{e34}
    \lim_{\hat{F}^{o}(x) \to +\infty}\hat{F}^{o}(x)F(\nabla(u(x)+\gamma_0\log \hat{F}^{o}(x)))=0.
\end{align}
\end{proposition}

\begin{proof}
Using Proposition \ref{P2}, we know that
\begin{align}\label{e13}
    0\leq \hat{F}^{o}(x)^{-\beta}e^{u(x)}\leq\frac{C}{\hat{F}^{o}(x)^{N}}\ \ \text{in\ }\mathbb{R}^{N}\backslash \mathcal{W}_{1}.
\end{align}
Denote by $\tilde{u}=t_0-u\geq0$, where $t_0:=\sup\limits_{\mathbb{R}^{N}}u$. From \eqref{e1} and the fact that $\hat{F}(\xi)=F(-\xi)$, we know that $\tilde{u}$ is a solution of
\begin{align}\label{e10}
    \mathrm{div}\left(\hat{F}^{N-1}(\nabla\tilde{u})D\hat{F}(\nabla\tilde{u})\right)=\hat{F}^{o}(x)^{-\beta}e^{u}\ \ \text{in\ }\mathbb{R}^{N}\backslash \{0\}.
\end{align}
In light of \eqref{e9}, we obtain $\tilde{u}\rightarrow +\infty$ as $\hat{F}^{o}(x)\rightarrow +\infty$. From \eqref{e13}, applying Theorem \ref{Th2} to \eqref{e10}, there exists a constant $d>1$, we get
\begin{align}\label{e11}
    -d\log \hat{F}^{o}(x)\leq u(x)\leq -\frac{1}{d}\log \hat{F}^{o}(x)
\end{align}
in a neighborhood of infinity.

Let $\gamma>0$ be such that
\begin{align}\label{e45}
    -\gamma=\liminf_{\hat{F}^{o}(x) \rightarrow +\infty}\frac{u(x)}{\log \hat{F}^{o}(x)}.
\end{align}
According to \eqref{e11}, we know that $\gamma$ exists.

Step 1. We claim that
\begin{align}\label{e15}
    u(x)+\gamma\log \hat{F}^{o}(x)\geq C\ \ \text{in\ }\mathbb{R}^{N}\backslash \mathcal{W}_{1},
\end{align}
for some constant $C$.

Set $\theta:=\inf\limits_{\partial{\mathcal{W}_{2}}}u$. Now, let $u_R(x)=\frac{u(Rx)-\theta}{\log R}$ and $f_R(x)=e^{u(Rx)}$ for $R>2$ and $\hat{F}^{o}(x)>\frac{1}{R}$. From \eqref{e11}, we have
\begin{align}\label{e14}
    \vert u_R(x)\vert\leq \frac{d\log \hat{F}^{o}(Rx)+C}{\log R}.
\end{align}
This implies $u_R$ is uniformly bounded in $L_{loc}^{\infty}(\mathbb{R}^{N}\backslash \{0\})$.

On the other hand, $u_R$ satisfies the equation
\begin{align}\label{e12}
    -Q_{N}u_R=\frac{R^{N-\beta}}{(\log R)^{N-1}}\hat{F}^{o}(x)^{-\beta}f_{R}\ \ \text{in\ }\mathbb{R}^{N}\backslash \overline{\mathcal{W}_{\frac{1}{R}}}.
\end{align}
By \eqref{e13}, we have $\hat{F}^{o}(x)^{-\beta}f_{R}\leq\frac{C}{R^{N-\beta}\hat{F}^{o}(x)^{N}}$ for $x\in \mathbb{R}^{N}\backslash \{0\}$, which implies that the right-hand side of \eqref{e12} is also uniformly bounded in $L_{loc}^{\infty}(\mathbb{R}^{N}\backslash \{0\})$. Applying the standard elliptic regularity
theory, we get $u_R$ is bounded in $C_{loc}^{1,\alpha}(\mathbb{R}^{N}\backslash \{0\})$ for some $\alpha\in (0,1)$. Therefore, by the Ascoli-Arzela's theorem there exists a subsequence $\{u_{R_i}\}$ such that
\begin{align*}
    R_i\rightarrow+\infty,\ \ u_{R_i}\rightarrow u_{\infty}\ \ \text{in\ }C^1_{loc}(\mathbb{R}^{N}\backslash \{0\}),
\end{align*}
with some $u_{\infty}\in C^1_{loc}(\mathbb{R}^{N}\backslash \{0\})$ satisfying $-Q_{N}u_\infty=0$ in $\mathbb{R}^{N}\backslash \{0\}$.

From \eqref{e14}, we see that $\vert u_\infty(x)\vert\leq C$ for any $x\neq 0$. It follows from Lemma \ref{L4} that $u_\infty$ is a constant. 

Next, we will determine this constant. Define
\begin{align*}
    \Gamma(R)=\inf_{2\leq \hat{F}^{o}(x)\leq R}\frac{u(x)-\theta}{\log \hat{F}^{o}(x)}.
\end{align*}
It is easy to know that $\Gamma(R)$ is non-increasing and $\Gamma(R)<0$ for large $R$ by \eqref{e9}. Since $u-\theta$ satisfies
\begin{align*}
    -Q_{N}(u-\theta)=\hat{F}^{o}(x)^{-\beta}e^{u}\geq0\ \ \text{in\ }\mathcal{W}_{R}\backslash\overline{\mathcal{W}_{2}}
\end{align*}
provide $R$ large enough, and
\begin{align*}
    u(x)-\theta\geq0\ \ \text{on\ }\partial{\mathcal{W}_{2}},\ \ u(x)-\theta\geq\Gamma(R)\log \hat{F}^{o}(x)\ \ \text{on\ }\partial{\mathcal{W}_{R}}.
\end{align*}
From the weak maximum principle, we see that $u(x)-\theta$ attains its minimum on the boundary of $\partial{\mathcal{W}_{R}}$, i.e.,
\begin{align*}
    u(x)-\theta\geq\inf_{y\in\partial{\mathcal{W}_{R}}}u(y)-\theta\ \ \text{for\ }x\in\mathcal{W}_{R}\backslash\overline{\mathcal{W}_{2}}.
\end{align*}
Thus,
\begin{align*}
    \frac{u(x)-\theta}{\log \hat{F}^{o}(x)}\geq\frac{\inf\limits_{y\in\partial{\mathcal{W}_{R}}}u(y)-\theta}{\log \hat{F}^{o}(x)}\geq\frac{\inf\limits_{y\in\partial{\mathcal{W}_{R}}}u(y)-\theta}{\log R}\ \ \text{for\ }x\in\mathcal{W}_{R}\backslash\overline{\mathcal{W}_{2}},
\end{align*}
which implies that
\begin{align*}
    \Gamma(R)=\inf_{x\in\partial{\mathcal{W}_{R}}}\frac{u(x)-\theta}{\log \hat{F}^{o}(x)}.
\end{align*}
According to the definition of $-\gamma$, we have $\lim\limits_{R\rightarrow +\infty}\Gamma(R)=-\gamma$. Since $\Gamma(R)$ is non-increasing, we get $\Gamma(R)\geq-\gamma$ for $R>2$. Hence,
\begin{align*}
u(x)\geq-\gamma\log \hat{F}^{o}(x)+\theta,\ \ \ \text{if}\ \ \hat{F}^{o}(x)\geq R,
\end{align*}
which together with $u^+\in L^{\infty}(\mathbb{R}^{N})$ gives \eqref{e15}. 

Step 2. We will show that
\begin{align}\label{e16}
    \lim_{ \hat{F}^{o}(x)\rightarrow +\infty}\frac{u(x)}{\log \hat{F}^{o}(x)}=-\gamma.
\end{align}
For large $R$, let $x_R\in\partial{\mathcal{W}_{1}}$ be such that $u_R(x_R)=\Gamma(R)$. By extracting a subsequence, we can assume that $x_{R_i}\rightarrow x_{\infty}$ on $\partial{\mathcal{W}_{1}}$. We can derive that 
\begin{align*}
    u_{\infty}(x_{\infty})=\lim_{R_i\rightarrow+\infty}u_{R_i}(x_{R_i})=\lim_{R_i\rightarrow+\infty}\Gamma(R_i)=-\gamma.
\end{align*}
Hence, from the uniqueness of $-\gamma$, we have
\begin{align*}
    u_{R}\rightarrow-\gamma\ \ \text{in}\ C_{loc}^{1}(\mathbb{R}^{N}\backslash \{0\}),\ \ \text{as\ }R\rightarrow+\infty,
\end{align*}
which implies \eqref{e16} hold. 

Step 3. In this step, we will improve the upper estimate of Proposition \ref{P2}. It follows from \eqref{e15} that
\begin{align*}
    \int_{\mathbb{R}^{N}}\hat{F}^{o}(x)^{-\beta}e^u dx\geq C\int_{\mathbb{R}^{N}\backslash \mathcal{W}_{1}}\hat{F}^{o}(x)^{-\beta} \hat{F}^{o}(x)^{-\gamma}dx.
\end{align*}
This together with the fact $\int_{\mathbb{R}^{N}}\hat{F}^{o}(x)^{-\beta}e^u dx<+\infty$ implies
\begin{align*}
    \gamma>N-\beta.
\end{align*}
Fix $0<\epsilon_0\ll1$ such that
\begin{align}\label{e51}
-\gamma_1:=-\gamma+\epsilon_0<-(N-\beta).
\end{align}
Hence, by \eqref{e16}, we have
\begin{align}\label{e18}
    u(x)\leq C-\gamma_1\log \hat{F}^{o}(x)\ \ \text{in\ }\mathbb{R}^{N}\backslash {\mathcal{W}_{1}},
\end{align}
for some constant $C$.

Step 4. In this step, we will construct a suitable supersolution to \eqref{e1} in an exterior domain.

Let $v(x)$ be a smooth, radially decreasing function on $x\in\mathbb{R}^{N}$ in the sense of Finsler metric $\hat{F}^{o}$. Since $v$ is radial, we rewrite
\begin{align*}
    v(x)=v(t),
\end{align*}
where $t=\hat{F}^{o}(x)$. Hence,
\begin{align*}
    -Q_{N}v=t^{1-N}(t^{N-1}(-v'(t))^{N-1})'.
\end{align*}
For $\hat{F}^{o}(x)\geq1$, it follows from $\eqref{e18}$ that $\hat{F}^{o}(x)^{-\beta}e^u\leq C_1 \hat{F}^{o}(x)^{-\beta-\gamma_1}$ for some $C_1>0$. In order to get a supersolution to \eqref{e1}, for some given $t_0\geq1$, let $v$ be a solution of
\begin{align*}
    t^{1-N}(t^{N-1}(-v'(t))^{N-1})'=C_1 t^{-\beta-\gamma_1}\ \ \text{for\ }t>t_0.
\end{align*}
Recalling that \eqref{e51}. Multiplying this equality by $t^{N-1}$ and then integrating from $t_0$ to $t$, we have
\begin{align}
    t^{N-1}(-v'(t))^{N-1}-t_0^{N-1}(-v'(t_0))^{N-1}&=C_1\int_{t_0}^{t} \tau^{N-\beta-\gamma_1-1}d\tau\nonumber\\
    &=\frac{C_1}{N-\beta-\gamma_1}\left(t^{N-\beta-\gamma_1}-t_0^{N-\beta-\gamma_1}\right)\label{e17}.
\end{align}
Denote by
\begin{align*}
    a=-\frac{N-\beta-\gamma_1}{C_1}t_0^{N-1}(-v'(t_0))^{N-1}+t_0^{N-\beta-\gamma_1},
\end{align*}
it follows from \eqref{e51} that
\begin{align*}
    a\geq t_0^{N-\beta-\gamma_1}>t^{N-\beta-\gamma_1}.
\end{align*}
\eqref{e17} can be rewritten as
\begin{align*}
    t^{N-1}(-v'(t))^{N-1}=\frac{C_1}{N-\beta-\gamma_1}\left(-a+t^{N-\beta-\gamma_1}\right),
\end{align*}
which is equivalent to
\begin{align*}
    -v'(t)=\left(\frac{C_1}{N-\beta-\gamma_1}\frac{-a+t^{N-\beta-\gamma_1}}{t^{N-1}}\right)^{\frac{1}{N-1}}.
\end{align*}
Integrating this equality from $t_0$ to $t$, we get
\begin{align}\label{e19}
    v(t)=-C_1^{\frac{1}{N-1}}(\gamma_1-N+\beta)^{\frac{1}{1-N}}\int_{t_0}^{t}\frac{\left(a-\tau^{N-\beta-\gamma_1}\right)^{\frac{1}{N-1}}}{\tau}d\tau+b,
\end{align}
for some $b\in\mathbb{R}$. 
Therefore, we obtain
\begin{align*}
    -Q_{N}v=C_1 \hat{F}^{o}(x)^{-\beta-\gamma_1}\geq \hat{F}^{o}(x)^{-\beta}e^u=-Q_{N}u\ \ \text{in\ }\mathbb{R}^{N}\backslash \overline{\mathcal{W}_{t_0}}.
\end{align*}
This means that $v$ is a supersolution to \eqref{e1} in an exterior domain.

Step 5. In this step, we show that for suitable choices of $a$ and $b$, $v(x)\geq u(x)$ holds for $x\in \mathbb{R}^{N}\backslash \overline{\mathcal{W}_{t_0}}$.

For each $\epsilon\in (0,\epsilon_0)$, we infer from \eqref{e16} that there exists large enough $R_\epsilon$ such that 
\begin{align}\label{e23}
    u(x)\leq -(\gamma-\epsilon)\log \hat{F}^{o}(x)\ \ \text{for\ }\hat{F}^{o}(x)\geq R_\epsilon.
\end{align}
Take
\begin{align}\label{e25}
    a=\frac{(\gamma-\epsilon)^{N-1}(\gamma_1+\beta-N)}{C_1},
\end{align}
since \eqref{e51}, we have $a\geq2t_0^{N-\beta-\gamma_1}$ provided $t_0$ large enough. By \eqref{e19}, we deduce that
\begin{align}\label{e20}
    v(x)&\geq-\left(\frac{a C_1}{\gamma_1-N+\beta}\right)^{\frac{1}{N-1}}\int_{t_0}^{\hat{F}^{o}(x)}\frac{1}{\tau}d\tau+b\nonumber\\
    &=-\left(\frac{a C_1}{\gamma_1-N+\beta}\right)^{\frac{1}{N-1}}\log \hat{F}^{o}(x)+C_2+b\nonumber\\
    &=-(\gamma-\epsilon)\log \hat{F}^{o}(x)+C_2+b,
\end{align}
where $C_2=(\gamma-\epsilon)\log t_0$. Now, for any $\epsilon\in (0,\epsilon_0)$, we may choose $b$ such that
\begin{align*}
    b\geq \max\{-C_2,\max_{\partial\mathcal{W}_{t_0}}u\}.
\end{align*}
From \eqref{e19}, we see that if $\hat{F}^{o}(x)=t_0$, then 
\begin{align*}
    v(x)=b\geq u(x),
\end{align*}
when $\hat{F}^{o}(x)\geq R_\epsilon$, it follows from \eqref{e23} and \eqref{e20} that
\begin{align}\label{e21}
    v(x)\geq -(\gamma-\epsilon)\log \hat{F}^{o}(x)\geq u(x).
\end{align}
Hence, $v(x)\geq u(x)$ on $\partial(\mathcal{W}_{R_\epsilon}\backslash \overline{\mathcal{W}_{t_0}})$. Recalling that $-Q_{N}v\geq-Q_{N}u$ in $\mathcal{W}_{R_\epsilon}\backslash \overline{\mathcal{W}_{t_0}}$. In view of the comparison principle and \eqref{e21}, we obtain
\begin{align}\label{e26}
    v(x)\geq u(x)\ \ \text{in\ }\mathbb{R}^{N}\backslash \mathcal{W}_{t_0}
\end{align}
for each $\epsilon\in (0,\epsilon_0)$.

Step 6. We claim that
\begin{align}\label{e24}
    u(x)+\gamma\log \hat{F}^{o}(x)\leq C\ \ \text{in\ }\mathbb{R}^{N}\backslash \mathcal{W}_{1},
\end{align}
for some constant $C$. By \eqref{e19} and \eqref{e25}, for any $\epsilon\in (0,\epsilon_0)$, we obtain
\begin{align}
    v(x)-b&=-\left(\frac{C_1}{\gamma_1-N+\beta}\right)^{\frac{1}{N-1}}\int_{t_0}^{\hat{F}^{o}(x)}\frac{\left(a-\tau^{N-\beta-\gamma_1}\right)^{\frac{1}{N-1}}}{\tau}d\tau\nonumber\\
    &=-\left(\frac{a C_1}{\gamma_1-N+\beta}\right)^{\frac{1}{N-1}}\int_{t_0}^{\hat{F}^{o}(x)}\frac{1}{\tau}d\tau+I\nonumber\\
    &=(-\gamma+\epsilon)(\log \hat{F}^{o}(x)-\log t_0)+I\label{e27},
\end{align}
where
\begin{align*}
    I=\left(\frac{C_1}{\gamma_1-N+\beta}\right)^{\frac{1}{N-1}}\int_{t_0}^{\hat{F}^{o}(x)}\frac{a^{\frac{1}{N-1}}-\left(a-\tau^{N-\beta-\gamma_1}\right)^{\frac{1}{N-1}}}{\tau}d\tau.
\end{align*}
Since $a\geq2t_0^{N-\beta-\gamma_1}$ and the function $s\mapsto s^{\frac{1}{N-1}}$ is concave for $s>0$, for $\tau>t_0$, we have
\begin{align*}
    0\leq a^{\frac{1}{N-1}}-\left(a-\tau^{N-\beta-\gamma_1}\right)^{\frac{1}{N-1}}&\leq\frac{\tau^{N-\beta-\gamma_1}}{N-1}\left(a-\tau^{N-\beta-\gamma_1}\right)^{\frac{2-N}{N-1}}\\
    &\leq\frac{\tau^{N-\beta-\gamma_1}}{N-1}t_0^{\frac{(N-\beta-\gamma_1)(2-N)}{N-1}},
\end{align*}
where we have used concave inequality $s_1^{\frac{1}{N-1}}-s_2^{\frac{1}{N-1}}\leq\frac{s_1-s_2}{N-1}s_2^{\frac{2-N}{N-1}}$ for any $s_1>s_2>0$. Hence,
\begin{align}
    I&\leq\left(\frac{C_1}{\gamma_1-N+\beta}\right)^{\frac{1}{N-1}}\int_{t_0}^{\hat{F}^{o}(x)}\frac{\tau^{N-\beta-\gamma_1-1}}{N-1}t_0^{\frac{(N-\beta-\gamma_1)(2-N)}{N-1}}d\tau\nonumber\\
    &=\left(\frac{C_1}{\gamma_1-N+\beta}\right)^{\frac{1}{N-1}}t_0^{\frac{(N-\beta-\gamma_1)(2-N)}{N-1}}\frac{\hat{F}^{o}(x)^{N-\beta-\gamma_1}-t_0^{N-\beta-\gamma_1}}{(N-1)(N-\beta-\gamma_1)}\nonumber\\
    &\leq\left(\frac{C_1}{\gamma_1-N+\beta}\right)^{\frac{1}{N-1}}\frac{t_0^{\frac{N-\beta-\gamma_1}{N-1}}}{(N-1)(\gamma_1-N+\beta)}\leq C,\label{e28}
\end{align}
where $C$ is independent of $\epsilon$. 
It follows from \eqref{e26}, \eqref{e27} and \eqref{e28} that
\begin{align*}
u(x)\leq v(x)\leq (-\gamma+\epsilon)(\log \hat{F}^{o}(x)-\log t_0)+C+b
\end{align*}
for $x\in\mathbb{R}^{N}\backslash \mathcal{W}_{t_0}$. Letting $\epsilon\rightarrow0$, we get \eqref{e24}.

Step 7. In view of \eqref{e15} and \eqref{e24}, we have
\begin{align}\label{e30}
    u(x)=-\gamma\log \hat{F}^{o}(x)+ H(x)\ \ \text{in\ }\mathbb{R}^{N}\backslash \mathcal{W}_{1},
\end{align}
where $H(x)\in L^{\infty}(\mathbb{R}^{N}\backslash \mathcal{W}_{1})$. Define
\begin{align*}
    u_m(x):=u(mx)+\gamma\log m,
\end{align*}
and
\begin{align*}
    H_m(x):=H(mx),
\end{align*}
for $m\in \mathbb{N}$. Then
\begin{align*}
    u_m(x)=-\gamma\log \hat{F}^{o}(x)+ H_m(x)
\end{align*}
and
\begin{align}
    -Q_Nu_m=m^{N-\beta} \hat{F}^{o}(x)^{-\beta}e^{u(mx)}=m^{N-\beta-\gamma}\hat{F}^{o}(x)^{-\beta-\gamma}e^{H_m(x)}.\label{e29}
\end{align}
Since $\gamma>N-\beta$, the right-hand side of \eqref{e29} is uniformly bounded in $L_{loc}^{\infty}(\mathbb{R}^{N}\backslash \{0\})$. Applying the standard elliptic regularity theory and Ascoli-Arzela’s theorem, we get $u_m\rightarrow u_{\infty}$ in $C_{loc}^{1}(\mathbb{R}^{N}\backslash \{0\})$ as $m\rightarrow+\infty$, where $u_{\infty}$ satisfies
\begin{align*}
    -Q_Nu_\infty=0\ \ \text{in\ }\mathbb{R}^{N}\backslash \{0\}.
\end{align*}
Consequently, passing to the limit in the decomposition of $u_m$ yields
\begin{align*}
    u_\infty(x)=-\gamma\log \hat{F}^{o}(x)+ H_\infty(x),
\end{align*}
with
\begin{align*}
    H_m(x)\rightarrow H_\infty(x)\ \ \text{in\ }C_{loc}^{1}(\mathbb{R}^{N}\backslash \{0\})\ \ \text{as}\ m\rightarrow+\infty.
\end{align*}
The uniform boundedness of $H(mx)$ gives $H_\infty(x)\in L^{\infty}(\mathbb{R}^{N})$. Using Lemma \ref{L4}, we obtain that $H_\infty(x)$ is a constant function.

From \eqref{e30},
\begin{align*}
    F(\nabla(u(x)+\gamma\log \hat{F}^{o}(x)))=F(\nabla H(x))=\frac{1}{m}F(\nabla H_m(y)),
\end{align*}
here $y=\frac{x}{m}$, this implies
\begin{align*}
    \sup_{\hat{F}^{o}(x)=m}\hat{F}^{o}(x)F(\nabla(u(x)+\gamma\log \hat{F}^{o}(x)))=\sup_{\hat{F}^{o}(y)=1}F(\nabla H_m(y)).
\end{align*}
Letting $m\rightarrow +\infty$, since $H_\infty(x)$ is a constant function, we get
\begin{align}\label{e31}
    \lim_{\hat{F}^{o}(x) \to +\infty}\hat{F}^{o}(x)F(\nabla(u(x)+\gamma\log \hat{F}^{o}(x)))=\sup_{\hat{F}^{o}(y)=1}F(\nabla H_\infty(y))=0.
\end{align}

Step 8. In this step, we will prove that $\gamma=\gamma_0$. By \eqref{e31}, we derive that
\begin{align*}
    \nabla u=-\gamma\nabla(\log \hat{F}^{o}(x))+o(\frac{1}{\hat{F}^{o}(x)})=-\gamma\frac{\nabla \hat{F}^{o}(x)}{\hat{F}^{o}(x)}+o(\frac{1}{\hat{F}^{o}(x)}).
\end{align*}
Hence, by Lemma \ref{L2.1}-(iv) and $F(\xi)=\hat{F}(-\xi)$ for $\xi\in\mathbb{R}^{N}$
\begin{align*}
    F(\nabla u)=\frac{\gamma}{\hat{F}^{o}(x)}+o(\frac{1}{\hat{F}^{o}(x)}).
\end{align*}
By Lemma \ref{L2.1}-(ii), (iii) and the divergence theorem, we have
\begin{align*}
    \int_{\mathcal{W}_{R}}\hat{F}^{o}(x)^{-\beta}e^u dx&=-\int_{\partial\mathcal{W}_{R}}F^{N-1}(\nabla u)\langle DF(\nabla u),\nu \rangle d\sigma(x)\\
    &=-\int_{\partial\mathcal{W}_{R}}F^{N-1}(\nabla u)\langle DF(\nabla u),\frac{\nabla \hat{F}^{o}(x)}{|\nabla \hat{F}^{o}(x)|} \rangle d\sigma(x)\\
    &=\frac{R}{\gamma}\int_{\partial\mathcal{W}_{R}}F^{N-1}(\nabla u)\langle DF(\nabla u),\frac{\nabla u+o(R^{-1})}{|\nabla \hat{F}^{o}(x)|} \rangle d\sigma(x)\\
    &=\frac{R}{\gamma}\int_{\partial\mathcal{W}_{R}}F^{N-1}(\nabla u) \frac{F(\nabla u)+o(R^{-1})}{|\nabla \hat{F}^{o}(x)|}  d\sigma(x)\\
    &=\int_{\partial\mathcal{W}_{R}}(\gamma R^{-1}+o(R^{-1}))^{N-1}\frac{1+o(1)}{|\nabla \hat{F}^{o}(x)|} d\sigma(x),
\end{align*}
where $\nu$ is the outer unit normal on $\partial\mathcal{W}_{R}$. Letting $R\rightarrow +\infty$, we obtain
\begin{align*}
    \int_{\mathbb{R}^{N}}\hat{F}^{o}(x)^{-\beta}e^u dx=\gamma^{N-1}\int_{\partial\mathcal{W}_{1}}\frac{1}{|\nabla \hat{F}^{o}(x)|} d\sigma(x)=N\kappa_N\gamma^{N-1},
\end{align*}
which means
\begin{align}\label{e32}
    \gamma=\left[\frac{\int_{\mathbb{R}^{N}}\hat{F}^{o}(x)^{-\beta}e^u dx}{N\kappa_N}\right]^{\frac{1}{N-1}}=\gamma_0.
\end{align}
Thus, combining \eqref{e30}, \eqref{e31} and \eqref{e32}, we can get the conclusions \eqref{e33} and \eqref{e34}. The proof is finished.
\end{proof}

\section{Geometric Shape of the Level Sets of $u$ and Classification of Solutions to Equation \eqref{e1}}

This section is devoted to studying the geometric shape of the level sets of $u$, and to completing the classification of solutions to Equation \eqref{e1}. To this end, we first prove the following lower-bound estimate of mass $\int_{\mathbb{R}^{N}}\hat{F}^{o}(x)^{-\beta}e^u dx$. 
\begin{lemma}\label{L2}
    Let $u$ be a solution of \eqref{e1}. Then
\begin{align}\label{e37}
\int_{\mathbb{R}^{N}}\hat{F}^{o}(x)^{-\beta}e^udx \geq N(\frac{N(N-\beta)}{N-1})^{N-1}\kappa_N.
\end{align}
\end{lemma}

\begin{proof}
    Let $t_0=\sup\limits_ {x\in\mathbb{R}^{N}}u(x)$. For $t\in\mathbb{R},$ let $\Omega_{t}=\left\{  x\in\mathbb{R}^{N}|u(x)>t\right\}.$

By the divergence theorem, we have
\begin{align}
\int_{\Omega_{t}}\hat{F}^{o}(x)^{-\beta}e^{u}dx &  =\int_{\Omega_{t}
}-\mathrm{div}(F^{N-1}(\nabla u)DF(\nabla u))dx\nonumber\\
&  =\int_{\partial\Omega_{t}}F^{N-1}(\nabla u)\Big\langle DF(\nabla u),\frac{\nabla u}{|\nabla u|}\Big\rangle
d\sigma(x)\nonumber\\
&  =\int_{\partial\Omega_{t}}\frac{F^{N}(\nabla u)}{|\nabla u|}d\sigma(x),\label{eq2}
\end{align}
the outer unit normal on $\partial\Omega_{t}$ is $-\frac{\nabla u}{|\nabla u|}$.
Using the co-area formula, we have%
\[
\int_{\Omega_{t}}\hat{F}^{o}(x)^{-\beta}e^{u}dx=\int_{t}^{\infty}e^{s}%
\int_{\partial\Omega s}\frac{\hat{F}^{o}(x)^{-\beta}}{|\nabla u|}%
d\sigma(x)ds.
\]
Hence,
\begin{equation*}
-\frac{d}{dt}\Big(\int_{\Omega_{t}}\hat{F}^{o}(x)^{-\beta}e^{u}dx\Big)=e^{t}\int_{\partial\Omega_{t}}\frac{\hat{F}^{o}(x)^{-\beta}}%
{|\nabla u|}d\sigma(x),
\end{equation*}
which combined with \eqref{eq2} and the co-area formula yields
\begin{align}
&  -\frac{d}{dt}\Big(\int_{\Omega_{t}}\hat{F}^{o}(x)^{-\beta}e^{u}dx\Big)^{\frac{N}{N-1}}\nonumber\\
&  =-\frac{N}{N-1}\Big(\int_{\Omega_{t}}\hat{F}^{o}(x)^{-\beta}e^{u}dx\Big)^{\frac{1}{N-1}}\frac{d}{dt}\Big(\int_{\Omega_{t}}\hat{F}^{o}
(x)^{-\beta}e^{u}dx\Big)\nonumber\\
&  =\frac{N}{N-1}\Big(\int_{\partial\Omega_{t}}\frac{F^{N}(\nabla u)}{|\nabla u|}d\sigma(x)\Big)^{\frac{1}{N-1}}e^{t}\int_{\partial\Omega_t}\frac{\hat{F}^{o}(x)^{-\beta}}{|\nabla u|}d\sigma(x).\label{eq4}
\end{align}
By using the weighted anisotropic isoperimetric inequality (\ref{eq1.3}) and
H\"{o}lder's inequality, we have
\begin{align}
\Big((N-\beta)(N\kappa_N)^{\frac{1}{N-1}}\int_{\Omega_{t}}\hat{F}^{o}(x)^{-\beta}dx\Big)^{\frac
{N-1}{N}}\leq &  \int_{\partial \Omega_t}\hat{F}^{o}(x)^{-\frac{N-1}{N}\beta}\hat{F}(\nu)d\sigma(x)\nonumber\\
= & \int_{\partial \Omega_t}\hat{F}^{o}(x)^{-\frac{N-1}{N}\beta}F(-\nu)d\sigma(x)\nonumber\\
= &  \int_{\partial\Omega_{t}}\frac{F(\nabla u)}{|\nabla u|}
\hat{F}^{o}(x)^{-\frac{N-1}{N}\beta}d\sigma(x)\nonumber\\
\leq &  \Big(\int_{\partial\Omega_{t}}\frac{F^{N}(\nabla u)}
{|\nabla u|}d\sigma(x)\Big)^{\frac{1}{N}}\Big(\int_{\partial\Omega_{t}
}\frac{\hat{F}^{o}(x)^{-\beta}}{|\nabla u|}d\sigma(x)\Big)^{\frac{N-1}{N}
}.\label{eq5}
\end{align}
Integrating the inequality (\ref{eq4}) from $-\infty$ to $t_0$, and using
(\ref{eq5}), we get%
\begin{align}
\Big(\int_{\mathbb{R}^{N}}\hat{F}^{o}(x)^{-\beta}e^{u}dx\Big)^{\frac
{N}{N-1}} &  =\int_{-\infty}^{t_0}-\frac{d}{dt}\Big(\int_{\Omega_{t}
}\hat{F}^{o}(x)^{-\beta}e^{u}dx\Big)^{\frac{N}{N-1}}dt\nonumber\\
&  =\frac{N}{N-1}\int_{-\infty}^{t_0}\Big(\int_{\partial\Omega_{t}}
\frac{F^{N}(\nabla u)}{|\nabla u|}d\sigma(x)\Big)^{\frac{1}{N-1}
}e^{t}\int_{\partial\Omega_t}\frac{\hat{F}^{o}(x)^{-\beta}}{|\nabla u
|}d\sigma(x)dt\nonumber\\
&  \geq\frac{N(N-\beta)}{N-1}(N\kappa_N)^{\frac{1}{N-1}}\int_{-\infty}^{t_0}e^{t}\int
_{\Omega_{t}}\hat{F}^{o}(x)^{-\beta}dxdt\nonumber\\
&  =\frac{N(N-\beta)}{N-1}(N\kappa_N)^{\frac{1}{N-1}}\int_{\mathbb{R}^{N}}\hat{F}^{o}(x)^{-\beta
}e^{u}dx\nonumber.
\end{align}
Consequently, it holds
\[
\int_{\mathbb{R}^{N}}\hat{F}^{o}(x)^{-\beta}e^{u}dx\geq N(\frac{N(N-\beta)}{N-1})^{N-1}\kappa_N.
\]
The proof is finished.
\end{proof}

\begin{remark}
    In the case that $\beta<0$, the above lower-bound estimate of mass cannot be obtained, even in the isotropic case $\hat{F}^{o}(x)=|x|$, due to the lack of the corresponding weighted isoperimetric inequality.
\end{remark}

Now, we establish the quantization result of the mass $\int_{\mathbb{R}^{N}}\hat{F}^{o}(x)^{-\beta}e^u dx$ as follows.

\begin{proposition}\label{P3}
    Let $u$ be a solution of \eqref{e1}. Then
\begin{align}
\int_{\mathbb{R}^{N}}\hat{F}^{o}(x)^{-\beta}e^udx = N(\frac{N(N-\beta)}{N-1})^{N-1}\kappa_N.\label{e36}
\end{align}
\end{proposition}

\begin{proof}
Given $r>0,$ for any $0<\varepsilon<r$, using the Pohozaev identity (Lemma \ref{L5}) for
anisotropic equation \eqref{e1} in $\mathcal{W}_{r}\backslash \mathcal{W}_{\varepsilon}$, which reads as
\begin{align}
&  (N-\beta)\int_{\mathcal{W}_{r}\backslash \mathcal{W}_{\varepsilon}}\hat{F}^{o}(x)^{-\beta
}e^{u}dx\nonumber\\
&  =r^{-\beta+1}\int_{\partial \mathcal{W}_{r}}\frac{e^{u}}{|\nabla \hat{F}^{o}
(x)|}d\sigma(x) -\frac{r}{N} \int_{\partial \mathcal{W}_{r}}%
\frac{F^{N}(\nabla u)}{|\nabla \hat{F}^{o}(x)|}d\sigma(x)\nonumber\\
&  +\int_{\partial \mathcal{W}_{r}}F^{N-1}(\nabla u)\langle DF(\nabla u),\nu \rangle\langle x,\nabla u \rangle d\sigma(x)\nonumber\\
&  -\varepsilon^{-\beta+1}\int_{\partial \mathcal{W}_{\varepsilon}}\frac{e^{u}%
}{|\nabla \hat{F}^{o}(x)|}d\sigma(x)+\frac{\varepsilon}{N}
\int_{\partial \mathcal{W}_{\varepsilon}}\frac{F^{N}(\nabla u)}{|\nabla
\hat{F}^{o}(x)|}d\sigma(x)\nonumber\\
&-\int_{\partial \mathcal{W}_{\varepsilon}}F^{N-1}(\nabla u)\langle DF(\nabla u),\nu \rangle\langle x,\nabla u \rangle d\sigma(x),\label{eq16}
\end{align}
where $\nu$ is the outer unit normal. We next claim that
\begin{align*}
\liminf_{\varepsilon\rightarrow0}\Big(\varepsilon^{-\beta+1}\int
_{\partial\mathcal{W}_{\varepsilon}}\frac{e^{u}}{|\nabla \hat{F}^{o}
(x)|}d\sigma(x)-\frac{\varepsilon}{N}\int_{\partial
\mathcal{W}_{\varepsilon}}\frac{F^{N}(\nabla u)}{|\nabla \hat{F}^{o}
(x)|}d\sigma(x)\\
+\int_{\partial \mathcal{W}_{\varepsilon}}F^{N-1}(\nabla u)\langle DF(\nabla u),\nu \rangle\langle x,\nabla u \rangle d\sigma(x)\Big)=0.
\end{align*}
From \eqref{e39} and Lemma \ref{L2.1}-(iv), we derive that
\begin{align}\label{e41}
F(\nabla \hat{F}^{o}(x))\leq\eta\vert\nabla \hat{F}^{o}(x)\vert=\eta\vert-\nabla \hat{F}^{o}(x)\vert\leq\frac{\eta}{\alpha}F(-\nabla \hat{F}^{o}(x))=\frac{\eta}{\alpha}\hat{F}(\nabla \hat{F}^{o}(x))=\frac{\eta}{\alpha}
\end{align}
and
\[
\hat{F}(\nabla u)=F(-\nabla u)\leq \eta\vert\nabla u\vert\leq \frac{\eta}{\alpha}F(\nabla u).
\]
In view of Lemma \ref{L2.1}-(iv) and the fact that $\left\langle
x,\xi\right\rangle \leq F^{o}(x)F(\xi)$ (or $\left\langle
x,\xi\right\rangle \leq \hat{F}^{o}(x)\hat{F}(\xi)$), we have
\begin{align*}
&\int_{\partial \mathcal{W}_{\varepsilon}}F^{N-1}(\nabla u)\langle DF(\nabla u),\nu \rangle\langle x,\nabla u \rangle d\sigma(x)\nonumber\\
&\leq\int_{\partial \mathcal{W}_{\varepsilon}}F^{N-1}(\nabla u)F^{o}(DF(\nabla u))\frac{F(\nabla \hat{F}^{o}(x))}{|\nabla \hat{F}^{o}(x)|}\hat{F}^{o}(x)\hat{F}(\nabla u)d\sigma(x)\nonumber\\
&\leq\varepsilon\frac{\eta^2}{\alpha^2}\int_{\partial \mathcal{W}_{\varepsilon}}\frac{F^{N}(\nabla u)}{|\nabla \hat{F}^{o}(x)|}d\sigma(x).
\end{align*}
Hence, we only need to prove that
\begin{align}
\liminf_{\varepsilon\rightarrow0}\Big(\varepsilon^{-\beta+1}\int
_{\partial\mathcal{W}_{\varepsilon}}\frac{e^{u}}{|\nabla \hat{F}^{o}%
(x)|}d\sigma(x)+\big(\frac{\eta^2}{\alpha^2}-\frac{1}{N}\big)\varepsilon\int_{\partial
\mathcal{W}_{\varepsilon}}\frac{F^{N}(\nabla u)}{|\nabla \hat{F}^{o}%
(x)|}d\sigma(x)\big)=0.\label{eq17}%
\end{align}
Indeed, if (\ref{eq17}) does not hold, then we can find some $\varepsilon_{0}\leq1$ and $\tau>0$ such that
\[
\varepsilon^{-\beta}\int
_{\partial\mathcal{W}_{\varepsilon}}\frac{e^{u}}{|\nabla \hat{F}^{o}
(x)|}d\sigma(x)+\big(\frac{\eta^2}{\alpha^2}-\frac{1}{N}\big)\int_{\partial
\mathcal{W}_{\varepsilon}}\frac{F^{N}(\nabla u)}{|\nabla \hat{F}^{o}
(x)|}d\sigma(x)\geq\frac{\tau}{\varepsilon},
\]
for any $\varepsilon\leq\varepsilon_{0}$. Integrating the above inequality from $0$ to $\varepsilon_{0}$, we have
\begin{align}
+\infty &  =\int_{0}^{\varepsilon_{0}}\frac{\tau}{s}ds\leq\int_{0}
^{\varepsilon_{0}}\Big(s^{-\beta}\int_{\partial\mathcal{W}_{s}}\frac
{e^{u}}{|\nabla \hat{F}^{o}(x)|}d\sigma(x)+\big(\frac{\eta^2}{\alpha^2}-\frac{1}{N}
\big)\int_{\partial\mathcal{W}_{s}}\frac{F^{N}(\nabla u)}{|\nabla
\hat{F}^{o}(x)|}d\sigma(x)\Big)ds\nonumber\\
&  =\int_{\mathcal{W}_{\varepsilon_{0}}}\hat{F}^{o}(x)^{-\beta}e^{u
}dx+\big(\frac{\eta^2}{\alpha^2}-\frac{1}{N}\big)\int_{\mathcal{W}
_{\varepsilon_{0}}}F^{N}(\nabla u)dx.\label{eq18}
\end{align}
According to assumption \eqref{e1}, we get
\begin{align*}
\int_{\mathcal{W}_{\varepsilon_{0}}}\hat{F}^{o}(x)^{-\beta}e^{u}dx<+\infty.
\end{align*}
By the divergence theorem and the fact that $\left\langle
x,\xi\right\rangle \leq F^{o}(x)F(\xi)$, we have
\begin{align*}
\int_{\mathcal{W}_{\varepsilon_{0}}}F^{N}(\nabla u)dx= &  \int
_{\partial\mathcal{W}_{\varepsilon_{0}}}u F^{N-1}(\nabla u)
\Big\langle DF(\nabla u),\frac{\nabla \hat{F}^{o}(x)}{|\nabla
\hat{F}^{o}(x)|}\Big\rangle d\sigma(x)\\
&  ~-\int_{\mathcal{W}_{\varepsilon_{0}}}u\mathrm{div}(F^{N-1}(\nabla
u)DF(\nabla u))dx\\
\leq &  \int_{\partial\mathcal{W}_{\varepsilon_{0}}}u F^{N-1}(\nabla
u)F^{o}(DF(\nabla u))\frac{F(\nabla \hat{F}^{o}(x))}{|\nabla
\hat{F}^{o}(x)|}d\sigma(x)+\int_{\mathcal{W}_{\varepsilon_{0}}}u \hat{F}^{o
}(x)^{-\beta}e^{u}dx\\
\leq & \frac{\eta}{\alpha}\int_{\partial\mathcal{W}_{\varepsilon_{0}}}\frac{u F^{N-1}(\nabla u)}{|\nabla\hat{F}^{o}(x)|}d\sigma(x)+\int_{\mathcal{W}_{\varepsilon_{0}}}u \hat{F}^{o}(x)^{-\beta}e^{u}dx,
\end{align*}
where we have used Lemma \ref{L2.1}-(iv) and \eqref{e41} in the last inequality. By Lemma \ref{L2.1}-(ii), the finite mass condition in (\ref{e1}) and Proposition \ref{P3.1}, we have
\begin{align*}
\int_{\mathcal{W}_{\varepsilon_{0}}}F^{N}(\nabla u)dx<+\infty.
\end{align*}
Therefore, (\ref{eq18}) is impossible since the last terms are finite. Thus the
claim (\ref{eq17}) is proved. Letting $\varepsilon\rightarrow0$ in
(\ref{eq16}), we then get
\begin{align}
 (N-\beta)\int_{\mathcal{W}_{r}}\hat{F}^{o}(x)^{-\beta
}e^{u}dx
&  =r^{-\beta+1}\int_{\partial \mathcal{W}_{r}}\frac{e^{u}}{|\nabla \hat{F}^{o}%
(x)|}d\sigma(x) -\frac{r}{N} \int_{\partial \mathcal{W}_{r}}%
\frac{F^{N}(\nabla u)}{|\nabla \hat{F}^{o}(x)|}d\sigma(x)\nonumber\\
&  +\int_{\partial \mathcal{W}_{r}}F^{N-1}(\nabla u)\langle DF(\nabla u),\nu \rangle\langle x,\nabla u \rangle d\sigma(x).\label{e35}
\end{align}
Recalling that $\lim\limits_{\hat{F}^{o}(x) \to +\infty}\hat{F}^{o}(x)F(\nabla(u(x)+\gamma_0\log \hat{F}^{o}(x)))=0$, we infer that
\begin{align*}
\langle x,\nabla u \rangle=-\gamma_0+o(1),\ \ F(\nabla u)=\frac{\gamma_0}{\hat{F}^{o}(x)}+o(\frac{1}{\hat{F}^{o}(x)})
\end{align*}
and
\begin{align*}
\langle DF(\nabla u),\nu \rangle=-\frac{1}{|\nabla \hat{F}^{o}(x)|}+o(1)
\end{align*}
uniformly for $x\in \partial \mathcal{W}_{r}$, as $r\to +\infty$. Thus, 
\begin{align*}
&-\frac{r}{N} \int_{\partial \mathcal{W}_{r}}%
\frac{F^{N}(\nabla u)}{|\nabla \hat{F}^{o}(x)|}d\sigma(x)+\int_{\partial \mathcal{W}_{r}}F^{N-1}(\nabla u)\langle DF(\nabla u),\nu \rangle\langle x,\nabla u \rangle d\sigma(x)\nonumber\\
\rightarrow&-\kappa_N\gamma^{N}_0+N\kappa_N\gamma^{N}_0=(N-1)\kappa_N\gamma^{N}_0,
\end{align*}
as $r\to +\infty$. Since $u(x)+\gamma_0\log \hat{F}^{o}(x)\in L^{\infty}(\mathbb{R}^{N}\backslash \mathcal{W}_{1})$, one has
\begin{align*}
e^{u}\leq C(\hat{F}^{o}(x))^{-\gamma_0},
\end{align*}
as $\hat{F}^{o}(x)\to +\infty$. By Lemma \ref{L2}, we know that $\gamma_0\geq \frac{N(N-\beta)}{N-1}>N-\beta$. Hence
\begin{align*}
r^{-\beta+1}\int_{\partial \mathcal{W}_{r}}\frac{e^{u}}{|\nabla \hat{F}^{o}(x)|}d\sigma(x)\leq Cr^{N-\beta-\gamma_0}\rightarrow 0,
\end{align*}
as $r\rightarrow+\infty$.

Consequently, letting $r\rightarrow+\infty$ in \eqref{e35}, we conclude
\begin{align*}
 (N-\beta)\int_{\mathbb{R}^{N}}\hat{F}^{o}(x)^{-\beta
}e^{u}dx=(N-1)\kappa_N\gamma^{N}_0.
\end{align*}
This together with the definition of $\gamma_0$ in \eqref{e50} implies \eqref{e36} holds. The proof is finished.
\end{proof}

Based on the above lower bound estimate (Lemma \ref{L2}) and quantization result (Proposition \ref{P3}) for the mass, we are now ready to give the proof of Theorem \ref{Th1}.

\begin{proof}
    [Proof of Theorem \ref{Th1}]Proposition \ref{P3} implies the equality in (\ref{e37}), and hence the equality holds in the weighted anisotropic isoperimetric inequality used in the proof of the Lemma \ref{L2}. From Lemma \ref{L3}, $\Omega_{t}=\left\{  x\in\mathbb{R}^{N}|u(x)>t\right\}$ must be a
Wulff ball with center at the origin, i.e., $\Omega_{t}=$ $\mathcal{W}%
_{R(t)}$ for some $R(t)>0$. This implies the solution $u$ to \eqref{e1} must be radial in the sense of Finsler metric $\hat{F}^{o}$.

Denote by
\[
\omega(\hat{F}^{o}(x)^{\frac{N-\beta}{N}})=u(\hat{F}^{o}(x))-N\log\Big(\frac{N-\beta}{N}\Big).
\]
It is easy to check that $\omega$ is also a radial function centered on the origin and satisfies
\begin{align*}
\begin{cases}
    -Q_{N}\omega=e^\omega& \text{in } \mathbb{R}^{N}, \\
    \int_{\mathbb{R}^{N}}e^\omega dx<+\infty.
\end{cases}
\end{align*}
Thanks to the classification result in \cite{CL}, we obtain the following family of radial solutions
\begin{align*}
\omega(\hat{F}^{o}(x)^{\frac{N-\beta}{N}})=\log\frac{N(\frac{N^2}{N-1})^{N-1}\lambda^N}{\left[1+\lambda^{\frac{N}{N-1}}\hat{F}^{o}(x)^{\frac{N-\beta}{N-1}}\right]^N}
\end{align*}
for all $\lambda>0$. Hence,
\[
u(x)=\log\frac{(\frac{N}{N-1})^{N-1}(N-\beta)^{N}\lambda^N}{\left[1+\lambda^{\frac{N}{N-1}}\hat{F}^{o}(x)^{\frac{N-\beta}{N-1}}\right]^N},
\]
and the proof is finished.
\end{proof}

\end{document}